\documentclass[preprint,12pt]{elsarticle}

\usepackage{amssymb,amsmath,amsfonts,bm,amsthm,mathrsfs}
\usepackage{algorithm,algorithmic}
\usepackage{mathtools}\mathtoolsset{showonlyrefs}
\usepackage{color}
\usepackage{soul}
\usepackage{enumitem}
\usepackage{listings}
\usepackage{matlab-prettifier}

\usepackage{hyperref}

\newtheorem{theorem}{Theorem}%
\newtheorem{corollary}{Corollary}%
\newtheorem{remark}{Remark}%

\journal{Nuclear Physics B}

\begin{document}

\begin{frontmatter}



\title{Verified error bounds for the singular values of structured matrices with applications to computer-assisted proofs for differential equations}


\author[inst1]{Takeshi Terao\corref{cor1}}
\author[inst1]{Yoshitaka Watanabe}
\author[inst2]{Katsuhisa Ozaki}
\cortext[cor1]{Corresponding author. Email: terao.takeshi.412@m.kyushu-u.ac.jp}
\affiliation[inst1]{organization={Research Institute for Information Technology, Kyushu University},
            addressline={744 Motooka, Nishi-ku}, 
            city={Fukuoka},
            postcode={819-0395}, 
            state={Fukuoka},
            country={Japan}}
\affiliation[inst2]{organization={Department of Mathematical Science, Shibaura Institute of Technology},
            addressline={307 Fukasaku, Minuma-ku}, 
            city={Saitama},
            postcode={337-8570}, 
            state={Saitama},
            country={Japan}}

\begin{abstract}
This paper introduces two methods for verifying the singular values of the structured matrix denoted by $R^{-H}AR^{-1}$, where $R$ is a nonsingular matrix and $A$ is a general nonsingular square matrix. 
The first of the two methods uses the computed factors from a singular value decomposition (SVD) to verify all singular values; the second estimates a lower bound of the minimum singular value without performing the SVD. 
The proposed approach for verifying all singular values efficiently computes tight error bounds. 
The method for estimating a lower bound of the minimum singular value is particularly effective for sparse matrices. These methods have proven to be efficient in verifying solutions to differential equation problems, that were previously challenging due to the extensive computational time and memory requirements.
\end{abstract}




\end{frontmatter}



\section{Introduction}

Assume a nonsingular square matrix $A$ and Hermitian positive definite matrix $B=B^H$ are given, where $B^H$ denotes the conjugate transpose of $B$. An upper triangular Cholesky factor of $B$, denoted by $R$, is such that $B=R^HR$.
Our goal in this paper is to find an enclosure for the following singular value:
\begin{align}
    \sigma_i:=\sigma_i(R^{-H}AR^{-1})\quad \text{and}\quad \underline{\sigma}_i\leq \sigma_i\leq \overline{\sigma}_i \quad\text{for all $i$},\label{eq:aim1}
\end{align}
where $\sigma_i(\cdot)$ denotes the $i$th largest singular value.
The proposed method produces the enclosure~\eqref{eq:aim1}.

However, there are cases in which it may not be practical to calculate all singular values due to the required computing time and memory.
Moreover, in the case of finding rigorous error bounds for the solution of ordinary and partial differential equations, only an enclosure for the spectral norm is needed (e.g.,~\cite{watanabe2020computer4,watanabe2023computer3}) such as:
\begin{align}
    \underline \sigma_{n}^{-1}\geq \|RA^{-1}R^{H}\|_2\geq\overline \sigma_{n}^{-1},\label{eq:aim2}
\end{align}
where $\sigma_{\min}(\cdot)$ denotes the minimum singular value.
In computer-assisted proofs of solutions to non-linear differential equations based on the Newton-type method, 
the spectral norm $\|RA^{-1}R^{H}\|_2$ in \eqref{eq:aim2} corresponds to an approximate operator norm of the infinite-dimensional operator that linearizes the problem around a suitable approximate solution; its inclusion \eqref{eq:aim2} plays an essential role in verifying the existence of a solution.
Details regarding applications of these values are given in~\cite[Chapter 3.3]{nakao2019numerical}.

When $A$ is Hermitian, $\sigma_i$ are equal to the absolute values of the eigenvalues of $R^{-H}AR^{-1}$.
For such cases, several verification methods have been proposed~\cite{yamamoto1980error,yamamoto1982error,oishi2001fast,rump2010verification,miyajima2010fast,rump2011verified,miyajima2012numerical,miyajima2014fast,rump2022verified,rump2023fast}.
Currently, problems involving large-scale matrices, even when sparse, are solved using supercomputers~\cite{hoshi2020posteriori,ozaki2021verified}.
For the case that $B$ is the identity matrix, multiple verification methods for singular values have been proposed~\cite{yamamoto2001simple,oishi2001fast,rump2011verified,miyajima2014verified,rump2023fast}.
However, there has been little discussion regarding how to verify the bounds of $\sigma_i$ for a general matrix $A$ (cf.~\cite[p.440]{nakao2019numerical}); moreover what has been proposed is not efficient.
In particular, there is no general method for large-scale sparse matrices.
This paper proposes verification methods for the problems~\eqref{eq:aim1} and \eqref{eq:aim2}.
The proposed methods have advantages in terms of accuracy, stability, and computation time.
In addition, the proposed method for \eqref{eq:aim2} applies to sparse matrices $A$ and $B$ and can be shown to be efficient.

The remainder of the paper is organized as follows:
Section~\ref{sec:pre} introduces the notation and describes useful estimates of the norm, eigenvalue, and singular value of matrices.
Section~\ref{sec:proposed} summarizes the proposed methodology:
Subsection \ref{subsec:deco} defines the matrix decomposition used to compute all $\sigma_i$;
Subsection \ref{subsec:veri} describes the proposed verification methods for the problems \eqref{eq:aim1} and \eqref{eq:aim2};
Section~\ref{sec:numex} provides computational results showcasing the performance and limitations of the proposed methods.
Section \ref{sec:PDE} discusses applications of \eqref{eq:aim1} and the effectiveness of the proposed method.

\section{Previous works}\label{sec:pre}
\subsection{Notation}

In this paper, $\Lambda(A)$ represents the spectrum (the set of eigenvalues) and $\Sigma(A)$ is the set of singular values of $A\in\mathbb{C}^{n\times n}$.
Suppose that $\lambda_i(A)\in\Lambda(A)$ and $\sigma_i(A)\in\Sigma(A)$ are ordered such that $|\lambda_1(A)|\leq\dots\leq |\lambda_n(A)|$ and $\sigma_1(A)\geq\dots\geq \sigma_n(A)$.
We define
\begin{align}
    \lambda_{\max}(A):=\max\{|\lambda|\ |\ \lambda\in\Lambda(A)|\},\quad \lambda_{\min}(A):=\min\{|\lambda|\ |\ \lambda\in\Lambda(A)|\},\label{eq:lamminmax}
\end{align}
and
\begin{align}
    \sigma_{\max}(A):=\sigma_1(A),\quad \sigma_{\min}(A):=\sigma_{n}(A).\label{eq:sigminmax}
\end{align}
Unless otherwise stated, the norm $\|\cdot\|$ for a vector or a matrix denotes the spectral norm $\|\cdot\|_2$.
$\kappa_2(A)$ for a nonsingular matrix $A$ denotes the condition number of $A$ such that $\kappa_2(A)=\|A\|\cdot\|A^{-1}\|$.
The number of positive, negative, and zero eigenvalues is denoted by $\mathrm{npe}(A)$, $\mathrm{nne}(A)$, and $\mathrm{nze}(A)$, respectively.
We write $A\succ 0$ ($A\succeq 0$) to denote that a square Hermitian matrix $A$ is positive (semi-) definite.

Let $\mathbb{IC}$ denote the set of complex intervals. In this paper, intervals are represented in boldface italics; for example, $\bm{A} \in \mathbb{IC}^{m \times n}$ denotes an interval matrix. An interval matrix is defined as follows:
\begin{align}
    \bm{A}=\{A\in\mathbb{C}^{m\times n}\ |\ \forall (i, j),\  a_{ij}\in\bm{a}_{ij}\in\mathbb{IC}\}.
\end{align}
For a real interval $\bm{a}$ containing $a\in\bm a$, the infimum $\mathrm{inf}(\bm{a})$ and supremum $\mathrm{sup}(\bm{a})$ are defined such that
\begin{align}
    \mathrm{inf}(\bm{a})\leq  a\leq\mathrm{sup}(\bm{a}),\quad \mathrm{inf}(\bm{a}), \mathrm{sup}(\bm{a})\in \bm{a}.
\end{align}
For a complex interval $\bm{a} \in \mathbb{IC}$, the infimum and supremum are denoted as:
\begin{align}
    \mathrm{inf}(\bm{a})&=\mathrm{inf}(\mathrm{Re}(\bm{a}))+\mathrm{inf}(\mathrm{Im}(\bm{a}))\cdot i,\\ 
    \mathrm{sup}(\bm{a})&=\mathrm{sup}(\mathrm{Re}(\bm{a}))+\mathrm{sup}(\mathrm{Im}(\bm{a}))\cdot i,
\end{align}
where $\mathrm{Re}(\cdot), \mathrm{Im}(\cdot)$, and $i$ are the real part, imaginary part, and imaginary unit, respectively.
Additionally, the midpoint $\mathrm{mid}(\bm{a})$ and radius $\mathrm{rad}(\bm{a})$ of an interval are given by
\begin{align}
    \mathrm{mid}(\bm{a}):=\frac{\mathrm{inf}(\bm{a})+\mathrm{sup}(\bm{a})}{2}\quad\text{and}\quad \mathrm{rad}(\bm{a}):=\frac{|\mathrm{sup}(\bm{a})-\mathrm{inf}(\bm{a})|}{2}.
\end{align}
Additional discussion is required for machine-interval operations but is omitted from this paper (for details, see \cite{moore2009introduction}).

\subsection{Estimation of singular values and the matrix norm}

Assume that the given matrix $A$ is nonsingular.
It is known that $\|A\|=\sigma_{\max}(A)$ and $\|A^{-1}\|=\sigma_{\min}(A)^{-1}$.
For the summation and product of matrices $A$ and $E$,
\begin{align}
    \sigma_i(A)-\|E\|\leq \sigma_i(A+E)\leq \sigma_i(A)+\|E\|\label{eq:sum}\\
    \sigma_i(A)\sigma_{\min}(E)\leq \sigma_i(AE)\leq \sigma_i(A)\|E\|\label{eq:prod}
\end{align}
are satisfied~\cite[Theorem 3.3.16]{horn1994topics}.
Similarly, it follows 
\begin{align}
    \sigma_{\min}(A)\sigma_{i}(E)\leq \sigma_i(AE)\leq \|A\|\sigma_{i}(E).\label{eq:prod2}
\end{align}
Next, we consider the estimations for approximate unitary matrix $X$.
Assume that $X^HX=I+F$, where $I$ is the identity matrix and $\|F\|\leq \epsilon<1$.
Then
\begin{align}
    \sqrt{1-\epsilon}\leq \sigma_{\min}(X),\quad \|X\|\leq \sqrt{1+\epsilon}
\end{align}
are satisfied~\cite{rump2011verified}.
Similarly, for a Hermitian positive definite $B\in\mathbb{C}^{n\times n}$, assume that $X_B$ satisfies $X_B^HBX_B=I+F$ and $R$ is the Cholesky factor as $B=R^HR$.
Then, if $\|F\|\leq \epsilon<1$,
\begin{align}
    \sqrt{1-\epsilon}\leq \sigma_{\min}(RX_B),\quad \|RX_B\|\leq \sqrt{1+\epsilon}\label{eq:borth}
\end{align}
are satisfied~\cite{miyajima2010fast}.

The spectral norms $\|A\|$ and $\||A|\|$ for Hermitian matrix $A$ can be bounded by the infinity norm $\|A\|_{\infty}$, i.e., 
let $e$ be the ones vector as $e=(1,1,\dots,1)^T$, then it follows
\begin{align}
    \|A\|\leq \||A|\|\leq \|A\|_{\infty}=\max_i(|A|e)_i.
\end{align}
Define $\mathrm{mag}(\bm A)=|\mathrm{mid}(\bm A)|+\mathrm{rad}(\bm A)$.
Then, for any Hermitian matrix $A\in\bm A$, it follows
\begin{align}
    \|A\|\leq \|\mathrm{mag}(\bm A)\|\leq \|\mathrm{mag}(\bm A)\|_{\infty}=\max_i(\mathrm{mag}(\bm A)e)_i.
\end{align}

For a general matrix $A$, because
\begin{align}
    \|A\|\leq \||A|\|=\sqrt{\||A|^T|A|\|},
\end{align}
it holds that
\begin{align}
    \|A\|_2&\leq \max_{i}\sqrt{(|A|^T(|A|e))_i}\\
    &\leq\max_i(\mathrm{mag}(\bm A^T)\cdot(\mathrm{mag}(\bm A)e))_i.
\end{align}

\subsection{Estimation of eigenvalues}

It is well-established that for any matrix $A \in \mathbb{C}^{n \times n}$:
\begin{align}
\lambda_{\min}(A) \geq \sigma_{\min}(A),\quad \lambda_{\max}(A) \leq \sigma_{\max}(A)\label{eq:eig1}
\end{align}
from \cite[Theorem~5.6.9]{horn2012matrix}, where the notation was introduced in \eqref{eq:lamminmax} and \eqref{eq:sigminmax}, and it follows
\begin{align}
\lambda_{i}(AB) = \lambda_{i}(BA)\label{eq:eig2}
\end{align}
for all $i$~\cite[p. 55]{horn2012matrix}. 
These relationships are fundamental to understanding the interplay between the eigenvalues and singular values of matrices.
Assume symmetric matrix $A$, nonzero vector $v$, and scalar $\alpha$. 
For the cases $\alpha>0$, $\alpha<0$, and $\alpha=0$, it follows
\begin{align}
    \begin{cases}
        \lambda_i(A)<\lambda_i(A+\alpha vv^H), & \alpha>0,\\
        \lambda_i(A)>\lambda_i(A+\alpha vv^H), & \alpha<0,\\
        \lambda_i(A)=\lambda_i(A+\alpha vv^H), & \alpha=0
    \end{cases}
\end{align}
for all $i$~\cite{golub1973some,bunch1978rank}, respectively.
From this, for symmetric positive definite matrix $B$ and scalar $\beta>0$,
\begin{align}
    \lambda_i(A)<\lambda_i(A+\beta B)=\lambda_i\left(A+\beta\sum_j\lambda_j(B)x_jx_j^H\right)
\end{align}
is satisfied, where $x_j$ is the eigenvector corresponding to $\lambda_j(B)$.

For the special case in which $A$ is Hermitian and $B$ is Hermite positive definite, Rump proposed a verification method for the minimum eigenvalue of $B^{-1}A$~\cite{rump2011verified}.
For $n$-by-$n$ Hermitian matrices $A$ and $0\prec B$, it holds that
\begin{align}
    \beta>0,\quad A^2-\beta B^2\succ 0 \quad \Longrightarrow \quad  \lambda_{\min}(B^{-1}A)^2\geq \beta.\label{eq:mineig2}
\end{align}
In particular,
\begin{align}
    0\prec A,\quad \beta>0,\quad A-\beta B\succ 0 \quad \Longrightarrow \quad  \lambda_{\min}(B^{-1}A)\geq \beta.\label{eq:mineig}
\end{align}
Using the methodologies and relationships previously discussed, we can estimate a lower bound for the smallest eigenvalue, which corresponds to the minimum singular value. The focus of this paper is on a general square matrix $A$. We intend to systematically apply these evaluations to understand and characterize the spectral properties of such matrices.

\subsection{Previous work regarding the issue}

Here, we consider interval matrices $\bm{A}$ and Hermitian $\bm{B}$.
We introduce the previous work to obtain a rigorous enclosure of $\sigma_{\min}:=\sigma_{\min}(\bm{R}^{-H}\bm{A}\bm{R}^{-1})$ in~\cite[p. 440]{nakao2019numerical}, where $\bm{R}$ is the interval Cholesky factor of $\bm{B}$.
Here, $\bm{R}$ is identified as the interval Cholesky factor of the matrix $\bm{B}$. For the implementation of Algorithm~\ref{alg:prework}, we employ MATLAB/INTLAB \cite{Ru99a}, a computational environment that facilitates interval arithmetic operations. Within this framework, the matrices $\bm{A}$ and $\bm{B}$ are represented using the {\tt intval} class provided by INTLAB, ensuring that our calculations adhere to the principles of interval arithmetic for rigorous numerical analysis.

\begin{algorithm}
\caption{Previous work  to obtain a rigorous enclosure of $\|RA^{-1}R^{H}\|$.}\label{alg:prework}
\begin{algorithmic}
\STATE{$\bm{R}=\mathrm{verchol}(\bm{B})$;\hfill Interval Cholesky decomposition for $B$~\cite{alefeld1993cholesky}}
\STATE{$\rho=\mathrm{norm}(\bm{R}\bm{A}^{-1}\bm{R}^{H})$;}
\end{algorithmic}
\end{algorithm}

Note that all operations must be carried out strictly.
A key function is $\mathrm{verchol}$, which computes the interval Cholesky factor of the matrix $\bm{B}$, as described in the works of Alefeld et al. \cite{alefeld1993cholesky,alefeld2009new}.
To compute the matrix $\bm{R}\bm A^{-1}\bm{R}^{H}$, we employ verified numerical computation techniques. 
Initially, we solve the linear equation $\bm{S}\supset \bm A^{-1}\bm{R}^{H}$ using the methods outlined by Rump in \cite{rump2013accurate1,rump2013accurate2}. Subsequently, the product $\bm{R}\bm{S}$ is calculated through interval matrix multiplication, following the efficient algorithms proposed in \cite{oishi2002fast,rump2012fast,ogita2005fast,ozaki2012fast}.
The $\mathrm{norm}$ function for an interval matrix returns the rigorous enclosure of the spectral norm~\cite{rump2011verified}.

\section{The proposed method}\label{sec:proposed}

In this section, we offer a new matrix decomposition for $A\in\mathbb{C}^{n\times n}$ and $B=B^H\in\mathbb{C}^{n\times n}$ to compute $\sigma_{j}(R^{-H}AR^{-1})$, where $B$ is positive definite and $R$ is an upper triangular Cholesky factor.
We then propose a verification method for $\sigma_{j}(R^{-H}AR^{-1})$ using approximate factors of the proposed decomposition.

\subsection{A new matrix decomposition}\label{subsec:deco}

The proposed method considers the matrix decomposition such that
\begin{align}
    AV_B=BU_B\Sigma,\quad V_B^HBV_B=U_B^HBU_B=I,\label{eq:GHSVD}
\end{align}
where $\Sigma\in\mathbb{R}^{n\times n}$ is the diagonal matrix with non-negative real numbers on the diagonal and $V_B, U_B\in\mathbb{C}^{n\times n}$.
Matrices $U_B$ and $V_B$ are called $B$-unitarity matrices when $V_B^HBV_B=U_B^HBU_B=I$ holds.
The factors of \eqref{eq:GHSVD} can be derived from the singular value decomposition of \( R^{-H}AR^{-1} \).  
Let the singular value decomposition be expressed as \( R^{-H}AR^{-1} = U\Sigma V^H \), where \(\Sigma \in \mathbb{R}^{n \times n}\) is a diagonal matrix with non-negative real numbers on its diagonal, and \(U\) and \(V\) are unitary matrices.  
Assuming \(U_B = R^{-1}U\) and \(V_B = R^{-1}V\), these matrices satisfy the following properties:  
\begin{itemize}[labelwidth=8em, leftmargin=9em]
    \item[$B$-unitarity:] \( U_B^HBU_B = U^HU = I \) and \( V_B^HBV_B = V^HV = I \),
    \item[Diagonality:] \( U_B^HAV_B = \Sigma \).
\end{itemize}
Note that this decomposition differs from the BSVD (Theorem 2 in \cite{van1976generalizing}), where \( A = BU_B\Sigma V_B^H \), because \( V_BV_B^H = B^{-1} \neq I \).

The diagonal element $\sigma_i=\Sigma_{ii}$ in \eqref{eq:GHSVD} satisfies
\begin{align}
    \sigma_i=\sigma_i(R^{-H}AR^{-1}).
\end{align}
This value (in particular, $\sigma_{\min}$) is essential in computer-assisted proofs for partial differential equations~\cite{nakao2019numerical}.
Thus, there is a demand for fast inclusion and accurate calculations.

Algorithm~\ref{alg:GHSVD} gives the procedure for calculating the proposed decomposition.
\begin{algorithm}
\caption{The proposed decomposition for $A$ and $B=B^H$ such that $AV_B= BU_B\Sigma$, $U_B^HBU_B=V_B^HBV_B=I$. If $A$ and $B$ are interval matrices, $\mathrm{mid}(\bm{A})$ and $\mathrm{mid}(\bm{B})$ should be given.}\label{alg:GHSVD}
\label{alg:buildtree}
\begin{algorithmic}
\STATE{$R=\mathrm{chol}(B)$;\hfill Cholesky decomposition for $B$}
\STATE{$[U,\Sigma, V]=\mathrm{svd}(R^{-H}AR^{-1})$;\hfill singular value decomposition for $R^{-H}AR^{-1}$}
\STATE{$U_B=R^{-1}U$; $V_B=R^{-1}V$;}\hfill backward substitutions
\RETURN $U_B, \Sigma, V_B$
\end{algorithmic}
\end{algorithm}

This decomposition can be extended to generalized eigenvalue decomposition, such as:
\begin{align}
    \left(\begin{array}{cc}
        O & A^H \\
        A & O
    \end{array}\right)
    X&=
    \left(\begin{array}{cc}
        B & O \\
        O & B
    \end{array}\right)
    X
    \left(\begin{array}{cc}
        \Sigma & O \\
        O & -\Sigma
    \end{array}\right)
    ,\label{eq:geig}\\
    X&=\frac{1}{\sqrt{2}}
    \left(\begin{array}{cc}
        V_B & V_B \\
        U_B & -U_B
    \end{array}\right).
\end{align}
This decomposition is applicable to verification methods for $\sigma_{\min}$.
Even if the given matrices are sparse, $R^{-H}AR^{-1}$ and $RA^{-1}R^{H}$ become dense.
Thus, for large-scale sparse matrices, an approximate $\sigma_{\min}$ cannot be obtained using Algorithm~\ref{alg:GHSVD} due to the required computing time and memory.
In such cases, from \eqref{eq:geig}, $\widehat \sigma_{\min}$ can be computed using eigensolvers for generalized eigenvalue decomposition.

For $A,B\in \mathbb{C}^{n\times n}$, it follows
\begin{align}
    \begin{pmatrix}
        \mathrm{Re}(A)&-\mathrm{Im}(A)\\
        \mathrm{Im}(A)&\mathrm{Re}(A)
    \end{pmatrix}\bar V_B=
    \begin{pmatrix}
        \mathrm{Re}(B)&-\mathrm{Im}(B)\\
        \mathrm{Im}(B)&\mathrm{Re}(B)
    \end{pmatrix}\bar U_B
    \begin{pmatrix}
        \Sigma&O\\
        O&\Sigma
    \end{pmatrix},\label{eq:exgsvd}
\end{align}
where
\begin{align}
    \bar U_B=
    \begin{pmatrix}
        \mathrm{Re}(U_B)&-\mathrm{Im}(U_B)\\
        \mathrm{Im}(U_B)&\mathrm{Re}(U_B)
    \end{pmatrix},\quad 
    \bar V_B=
    \begin{pmatrix}
        \mathrm{Re}(V_B)&-\mathrm{Im}(V_B)\\
        \mathrm{Im}(V_B)&\mathrm{Re}(V_B)
    \end{pmatrix}.
\end{align}
Here, $\bar U_B$ and $\bar V_B$ have $B$-orthogonality and the diagonality such that
\begin{align}
    \bar U_B^T\begin{pmatrix}
        \mathrm{Re}(A)&-\mathrm{Im}(A)\\
        \mathrm{Im}(A)&\mathrm{Re}(A)
    \end{pmatrix}\bar V_B&=
    \begin{pmatrix}
        \Sigma&O\\
        O&\Sigma
    \end{pmatrix},\\
    \bar U_B^T\begin{pmatrix}
        \mathrm{Re}(B)&-\mathrm{Im}(B)\\
        \mathrm{Im}(B)&\mathrm{Re}(B)
    \end{pmatrix}\bar U_B&=
    \begin{pmatrix}
        I&O\\
        O&I
    \end{pmatrix},\\
    \bar V_B^T\begin{pmatrix}
        \mathrm{Re}(B)&-\mathrm{Im}(B)\\
        \mathrm{Im}(B)&\mathrm{Re}(B)
    \end{pmatrix}\bar V_B&=
    \begin{pmatrix}
        I&O\\
        O&I
    \end{pmatrix}.
\end{align}
From these observations~\eqref{eq:geig} and \eqref{eq:exgsvd}, the proposed matrix decomposition \eqref{eq:GHSVD} can be reduced to a real symmetric generalized eigenvalue problem. However, it should be noted that the problem size may increase by up to four times.

\subsection{Verification methods}\label{subsec:veri}

In this section, we propose verification methods for all $\sigma_i$ and a lower bound $\sigma_{\min}$.
With respect to the verification method for all $\sigma_i$, we envisage the cases in which the given matrices are dense or ill-conditioned.
The verification methods for a lower bound $\sigma_{\min}$ are applicable to large-scale sparse matrices.

\subsubsection{For all singular values}

We propose the verification method for the singular values $\sigma\in\mathbb{R}^n$ from~\eqref{eq:GHSVD}.
Assume that the approximations are obtained by Algorithm~\ref{alg:GHSVD} from given matrices $A,B=B^H\in\mathbb{C}^{n\times n}$ such that
\begin{align}
    A\widehat V_B\approx B\widehat U_B\widehat \Sigma,\quad \widehat U_B^HB\widehat U_B\approx I,\quad \widehat V_B^HB\widehat V_B\approx I.\label{eq:app}
\end{align}
The proposed method produces $\widehat \sigma_i$ and error bounds using numerical computations for $1\leq i\leq n$.

\begin{theorem}\label{thm:veri}
    Let $A\in\mathbb{C}^{n\times n}$ and $B=B^H\in\mathbb{C}^{n\times n}$ be given, and $\widehat U_B, \widehat \Sigma=\mathrm{diag}(\widehat \sigma_i)$, and $\widehat V_B$ be approximate results of Algorithm~\ref{alg:GHSVD}. Define 
    \begin{align}
        \delta = \|\widehat U^{H}_BA\widehat V_B-\widehat \Sigma\|\label{eq:assume1}
    \end{align}
    Assume that
    \begin{align}
        \|\widehat U_B^HB\widehat U_B-I\|\leq \alpha <1,\quad \|\widehat V_B^HB\widehat V_B-I\|\leq \beta <1.\label{eq:assume2}
    \end{align}
    Then, 
    \begin{align}
        \frac{\widehat \sigma_i-\delta}{\sqrt{(1+\alpha)(1+\beta)}}\leq \sigma_i\leq \frac{\widehat \sigma_i+\delta}{\sqrt{(1-\alpha)(1-\beta)}}\label{eq:thm1}
    \end{align}
    is satisfied for $1\leq i\leq n$.
\end{theorem}

\begin{proof}
Assume that $R$ is the Cholesky factor of $B$.
From \eqref{eq:prod}, \eqref{eq:prod2}, and
\begin{align}
    R^{-H}AR^{-1}=(R\widehat U_B)^{-H}\widehat U_B^HA\widehat V_B(R\widehat V_B)^{-1},
\end{align}
it follows
\begin{align}
    \frac{\sigma_i(\widehat U_B^HA\widehat V_B)}{\|R\widehat U_B\|\cdot\|R\widehat V_B\|}
    \leq \sigma_i(R^{-H}AR^{-1})
    \leq\frac{\sigma_i(\widehat U_B^HA\widehat V_B)}{\sigma_{\min}(R\widehat U_B)\cdot\sigma_{\min}(R\widehat V_B)}.
\end{align}

Then, from~\eqref{eq:borth} and \eqref{eq:assume2},
\begin{align}
    \frac{\sigma_i(\widehat U_B^HA\widehat V_B)}{\sqrt{(1+\alpha)(1+\beta)}}
    \leq \sigma_i(R^{-H}AR^{-1})
    \leq\frac{\sigma_i(\widehat U_B^HA\widehat V_B)}{\sqrt{(1-\alpha)(1-\beta)}}\label{eq:prrof1}
\end{align}
is satisfied.
Next, from \eqref{eq:sum} and \eqref{eq:assume1}, we obtain
\begin{align}
    \widehat \sigma_i-\delta\leq \sigma_i(\widehat U_B^HA\widehat V_B)\leq \widehat \sigma_i+\delta.\label{eq:prrof2}
\end{align}
From \eqref{eq:prrof1} and \eqref{eq:prrof2}, \eqref{eq:thm1} can be obtained.
\end{proof}

\begin{corollary}
Assume that \eqref{eq:assume1} and \eqref{eq:assume2} are satisfied and $\widehat\sigma_i-\delta>0$,
\begin{align}
    \frac{\sqrt{(1-\alpha)(1-\beta)}}{\widehat \sigma_i+\delta}\leq \sigma_i(RA^{-1}R^{H})\leq \frac{\sqrt{(1+\alpha)(1+\beta)}}{\widehat \sigma_i-\delta}
\end{align}
is satisfied.
\end{corollary}
\begin{proof}
When $\sigma_i(R^{-H}AR^{-1})>0$, $\sigma_i(R^{-H}AR^{-1})^{-1}\in\Sigma(RA^{-1}R^H)$ is satisfied.
\end{proof}

The following is an algorithm for verifying the singular values~\eqref{eq:aim1} based on Theorem~\ref{thm:veri}.

\begin{algorithm}
\caption{The proposed verification method for singular values $\sigma_{i}(R^{-H}AR^{-1})$, $1\leq i\leq n$ from $\bm{A}\ni A$, $\bm{B}=\bm{B}^H\ni R^HR$, and factors from Algorithm ~\ref{alg:GHSVD}.}\label{alg:veriGHSVD}
\begin{algorithmic}
\REQUIRE $\bm{A}, \bm{B}, \widehat U_B, \widehat \Sigma, \widehat V_B$
\STATE $\delta=\mathrm{NrmBnd}(\widehat U_B^H\bm{A}\widehat V_B-\widehat \Sigma)$;
\STATE{$\alpha=\mathrm{NrmBnd}(\widehat U_B^H\bm{B}\widehat U_B-I)$; $\beta=\mathrm{NrmBnd}(\widehat V_B^H\bm{B}\widehat V_B-I)$;}
\IF{$\alpha<1$ and $\beta<1$}
    \FOR{$1\leq i\leq n$}
        \STATE{$\underline{\sigma}_i=\mathrm{inf}\left(\displaystyle\frac{\widehat\sigma_i-\delta}{\sqrt{(1+\alpha)(1+\beta)}}\right)$; $\overline{\sigma}_i=\displaystyle\mathrm{sup}\left(\frac{\widehat \sigma_i+\delta}{\sqrt{(1-\alpha)(1-\beta)}}\right)$;}
    \ENDFOR
\ENDIF
\RETURN $\underline{\sigma},\overline{\sigma}\in\mathbb{R}^n$
\end{algorithmic}
Note that $\underline{\sigma}_i$ and $\overline{\sigma}_i$ satisfy \eqref{eq:aim1}.
It should also be noted that the required matrices can be obtained by Algorithm~\ref{alg:GHSVD} with $\mathrm{mid}(\bm A)$ and $\mathrm{mid}(\bm B)$ using floating-point arithmetic.
On the other hand, for this algorithm, the interval class matrices $\bm A$ and $\bm B$ are required.
The function $\mathrm{NrmBnd}(\cdot)$ returns a rigorous upper bound of the spectral norm.
\end{algorithm}

The cost of the proposed method flows from the way in which an upper bound $\delta$ is computed.
Note that the computation of $\widehat U_B^HB\widehat U_B$ is duplicated in the computations of $\delta$ and $\alpha$.
In addition, the computation of $B\widehat U_B$ is duplicated in the computations of $A\widehat V_B-B\widehat U_B\widehat \Sigma$ and $\widehat U_B^HB\widehat U_B$.
Thus, our proposed implementation innovations can reduce calculation costs.

Next, $\delta$ can be analyzed as
\begin{align}
    \|\widehat U^{H}_BA\widehat V_B-\widehat \Sigma\|=&\|\widehat U^{H}_B(A\widehat V_B-B\widehat U_B\widehat \Sigma)+(\widehat U^H_BB\widehat U_B-I)\widehat \Sigma\|&=\delta_1 \label{eq:delta1}\\
    \leq &\|\widehat U^{H}_B(A\widehat V_B-B\widehat U_B\widehat \Sigma)\|+\alpha\cdot \widehat\sigma_{\max}&=\delta_2\label{eq:delta2}\\
    \leq &\|\widehat U^{H}_B\|\|A\widehat V_B-B\widehat U_B\widehat \Sigma\|+\alpha\cdot \widehat\sigma_{\max}&=\delta_3 &.\label{eq:delta3}
\end{align}
Based on a comparison of the computations of $\delta_1$ and $\delta_2$, we can reduce the memory cost since the computed result of $\widehat U^H_BB\widehat U_B$ can be removed after computing $\alpha$.
From~\eqref{eq:delta3}, the computation cost of one matrix multiplication can be reduced to compute an upper bound of $\delta$.
Note that $\|\widehat U_B^H\|$ can be bounded by $\mathcal{O}(n^2)$ flops.
However, the overestimation of $\delta_3$ is greater than that of the others.

Verification is based on matrix multiplications; however, the cost depends on such factors as whether the given matrices $A$ and $B$ are sparse.
When $A$ and $B$ are sparse, the cost of the verification is nearly 2 and 3 dense matrix multiplications using \eqref{eq:delta2} and \eqref{eq:delta3}, respectively.
On the other hand, when $A$ and $B$ are dense, the cost of the verification is nearly 5 and 6 dense matrix multiplications using \eqref{eq:delta2} and \eqref{eq:delta3}.
Moreover, there are various interval matrix multiplication algorithms~\cite{rump1999fast,ogita2005fast,ozaki2011tight,ozaki2012fast,rump2012fast,uchino2023inclusion}, where the verification method can achieve high performance with respect to speed, accuracy, or stability.

\subsubsection{For the minimum singular value}

We next consider the verification method based on \eqref{eq:geig}.
While the verification method based on Theorem~\ref{thm:veri} can be used to verify all singular values, it is not applicable to large-scale sparse matrices since, even if matrices $A$ and $B$ are sparse, the factors of the proposed decomposition $U_B$ and $V_B$ become dense.
To solve this problem, we propose a method that verifies only the minimum singular value.
Here, suppose that  
\begin{align}
    \bar A=\left(\begin{array}{cc}
        O & A^H \\
        A & O
    \end{array}\right),\quad 
    \bar B=
    \left(\begin{array}{cc}
        B & O \\
        O & B
    \end{array}\right),\quad 
    \bar \Sigma=
    \left(\begin{array}{cc}
        \Sigma & O \\
        O & -\Sigma
    \end{array}\right)
\end{align}
and consider the eigenvalue $\lambda_{\min}(\bar B^{-1}\bar A)$, which is equal to $\sigma_{\min}$.
\begin{theorem}\label{thm:sp}
    For matrices $A, 0\prec B=B^H\in\mathbb{C}^{n\times n}$, and real scalar $\theta>0$, 
    if $A^HA-\theta B^2$ and $AA^H-\theta B^2$ are positive definite, $\sigma_{\min}\geq \sqrt{\theta}$ is satisfied.
\end{theorem}
\begin{remark}
    The determination of whether a matrix is positive definite must be rigorous.
    For more detail, see \cite{rump2006verification}.
    INTLAB also supports a function to determine positive definiteness named {\tt isspd}.
\end{remark}
\begin{proof}
    From
    \begin{align}
        \bar A^2=\left(\begin{array}{cc}
            A^HA & O \\
            O & AA^H
        \end{array}\right),\quad 
        \bar B^2=
        \left(\begin{array}{cc}
            B^2 & O \\
            O & B^2
        \end{array}\right)
    \end{align}
    and \eqref{eq:mineig2}, the theorem is proved.
\end{proof}

\begin{algorithm}
\caption{The proposed verification for $\bm{A}$ and $\bm{B}=\bm{B}^H$ returning a lower bound of $\sigma_{\min}$.}\label{alg:veriGHSVD_type2}
\begin{algorithmic}
\REQUIRE $\bm{A}, \bm{B}, 0<\epsilon<1$
\STATE{Define $A=\mathrm{mid}(\bm A)$ and $B=\mathrm{mid}(\bm B)$}
\STATE{Compute $\widetilde \theta_1\approx \lambda_{\min}(B^{-2}A^HA)$ and $\widetilde \theta_2\approx \lambda_{\min}(B^{-2}AA^H)$;}
\STATE{$\theta = (1-\epsilon)\cdot\min\left(\widetilde\theta_1,\widetilde\theta_2\right)$;}
\IF{$\mathrm{isspd}(\bm{A}^H\bm{A}-\theta \bm{B}^2)$ and $\mathrm{isspd}(\bm{A}\bm{A}^H-\theta \bm{B}^2)$ are true} 
\STATE{$\underline{\sigma}_{\min}=\inf(\sqrt{\theta})$;}
\ENDIF
\RETURN $\underline{\sigma}_{\min}$
\end{algorithmic}
\end{algorithm}

The advantage of Theorem~\ref{thm:sp} is that the method can be applied to large sparse matrices.
While it is true that calculations $A^TA$, $AA^T$, and $B^2$ introduce fill-in, resulting in denser matrices, in many practical cases where \( A \) and \( B \) are sufficiently sparse, the resulting matrices still exhibit sparsity.
On the other hand, the Gram matrix of the given matrices makes it difficult to apply the theorem to ill-conditioned problems.

Next, we focus on the number of positive (or negative) eigenvalues.
The generalized eigenvalue problem
\begin{align}
    \bar A\bar X=\bar B\bar X\bar \Sigma
\end{align}
has $n$ positive eigenvalues from \eqref{eq:geig}.
A verified numerical computation method for generalized eigenvalue problems with sparse matrices is known~\cite{yamamoto2001simple}, and the proposed method reduces the computational cost by exploiting the structure of the matrices.

\begin{theorem}\label{thm:type3}
    Let $A\in\mathbb{C}^{n\times n}$, $0\prec B=B^H\in\mathbb{C}^{n\times n}$, and the real constant $\theta>0$ be given.
    Define
    \begin{align}
        G(\theta)=
        \left(\begin{array}{cc}
            \theta B & A^H \\
            A & \theta B
        \end{array}\right) .
    \end{align}
    If $\mathrm{npe}(G(\theta))=n$, it holds that $\sigma_{\min}\geq\theta$.
\end{theorem}
\begin{proof}
    For $\Lambda(A)$ denoting the spectrum of $A$,
    it shows that
    \begin{align}
        \Lambda(\bar B^{-1}\bar A)=\Lambda(\bar R^{-H}\bar A\bar R^{-1})=\{\sigma_1,\dots,\sigma_n,-\sigma_n,\dots,-\sigma_1\},
    \end{align}
    where $\bar R$ is the upper triangular Cholesky factor of $\bar B$.
    For $\mathrm{npe}(A)$, 
    it holds that
    \begin{align}
        \mathrm{npe}(\bar A+\theta \bar B)=\mathrm{npe}(\bar R^H(\bar R^{-H}\bar A\bar R^{-1}+\theta I)\bar R)=\mathrm{npe}(\bar R^{-H}\bar A\bar R^{-1}+\theta I)
    \end{align}
    from Sylvester's law of inertia.
    Because
    \begin{align}
        \Lambda(\bar R^{-H}\bar A\bar R^{-1}+\theta I)=\{\sigma_1+\theta,\dots,\sigma_n+\theta,-\sigma_n+\theta,\dots,-\sigma_1+\theta\},
    \end{align}
    it is true that $\sigma_{\min}>|\theta|$ if $\mathrm{npe}(G(\theta))=n$.
\end{proof}

For counting the number of positive eigenvalues of $G(\theta)$, $LDL^T$ decomposition $G(\theta)=LDL^H$ is often used~\cite{yamamoto2001simple}, where $L\in\mathbb{C}^{n\times n}$ is lower triangular and $D\in\mathbb{C}^{n\times n}$ is block diagonal with diagonal blocks of dimension 1 or 2~\cite{bunch1977some}.
Thus, the number of positive eigenvalues of $G(\theta)$ and that of $D$ are the same.
However, the possibility of rounding errors in the calculation of the $LDL^T$ decomposition should be noted.

\begin{theorem}\label{thm:type3-2}
    Let $G(\theta)\in\mathbb{C}^{2n\times 2n}$ be given from Theorem~\ref{thm:type3}.
    Define the real constants $\tau>0$ and $\delta>0$ to satisfy
    \begin{align}
        \|G(\theta)+\tau I- \widehat L\widehat D\widehat L^H\|=:\delta,
    \end{align}
    where $\widehat L, \widehat D$ are the approximate $LDL^T$ factors.
    If $\mathrm{npe}(\widehat D)=n$ and $\tau>\delta$, it holds that $\mathrm{npe}(G(\theta))=n$.
\end{theorem}
\begin{proof}
    Because $\bar B$ is positive definite, it holds that
    \begin{align}
        \lambda_i(\bar A)< \lambda_i(\bar A+\theta \bar B)=\lambda_i\left(\bar A+\theta \sum_{j=1}^n\lambda_j(\bar B)x_jx_j^H\right)=\lambda_i(G(\theta)),
    \end{align}
    where $x_j$ is the  eigenvector corresponding to the eigenvalue $\lambda_j(\bar B)$. 
    From $$\Lambda(\bar A)=\{\sigma_1(A),\dots,\sigma_n(A),-\sigma_n(A),\dots,-\sigma_1(A)\},$$
    it follows $\mathrm{npe}(G(\theta))\geq n$.
    From
    \begin{align}
        \lambda_i(G(\theta))+\tau \leq \lambda_{i}(\widehat L\widehat D\widehat L^H)+\delta
    \end{align}
    and $\tau>\delta$, it holds
    \begin{align}
        n\leq \mathrm{npe}(G(\theta))\leq \mathrm{npe}(\widehat L\widehat D\widehat L^H)=\mathrm{npe}(\widehat D).
    \end{align}
    Thus, if $\mathrm{npe}(\widehat D)=n$, it holds that $\mathrm{npe}(G(\theta))=n$.
\end{proof}

\href{}{\begin{algorithm}
\caption{The proposed verification for $\bm{A}$ and $\bm{B}=\bm{B}^H$ returning a lower bound of $\sigma_{\min}$.}\label{alg:veriGHSVD_type3}
\begin{algorithmic}
\REQUIRE $\bm{A}, \bm{B}, 0<\epsilon<1$
\STATE{Define $\bar A=\left(\begin{array}{cc}
    0 & \mathrm{mid}(\bm{A}^H) \\
    \mathrm{mid}(\bm{A}) & 0
\end{array}\right)$ and $\bar B=\left(\begin{array}{cc}
    \mathrm{mid}(\bm{B}) & O \\
    O & \mathrm{mid}(\bm{B})
\end{array}\right)$.
}
\STATE{Compute $\widetilde \theta\approx \lambda_{\min}(\bar B^{-1}\bar A)$.}
\STATE{$\theta \approx (1-\epsilon)\widetilde\theta$;}
\STATE{$\bm{G}=\left(\begin{array}{cc}
    \theta \bm{B} & \bm{A}^H \\
    \bm{A} & \theta \bm{B}
\end{array}\right)$;}
\IF{$\mathrm{npe}(\bm{G})=n$} 
\STATE{$\underline{\sigma}_{\min}=\theta$;}
\ENDIF
\RETURN $\underline{\sigma}_{\min}$
\end{algorithmic}
\end{algorithm}}

Finally, a verification method for an upper bound of the minimum singular value $\overline \sigma_{\min}$ is presented.
It is assumed that a lower bound of the minimum singular value $\underline \sigma_{\min}$ has been obtained.
In this case, it holds that
\begin{align}
    \sigma_{\min}\leq\underline \sigma_{\min}+\frac{1}{\sqrt{\sigma_{\min}(B)}}\cdot \frac{\|\bar A\widehat x-\underline\sigma_{\min}\bar B\widehat x\|}{\sqrt{\widehat x^H \bar B\widehat x}}=:\overline \sigma_{\min},
\end{align}
where $\widehat x$ is the approximate eigenvector corresponding to the approximate minimum eigenvalue of $\bar Ax=\lambda \bar Bx$.
This is because, from \cite{miyajima2010fast},
\begin{align}
    |\lambda_{\min}(\bar B^{-1}\bar A)-\underline\sigma_{\min}|
    &\leq \frac{\|\bar B^{-1}\bar A\widehat x-\underline\sigma_{\min}\widehat x\|}{\sqrt{\widehat x^H \bar B\widehat x}}\\
    &\leq \frac{1}{\sqrt{\sigma_{\min}(B)}}\cdot \frac{\|\bar A\widehat x-\underline\sigma_{\min}\bar B\widehat x\|}{\sqrt{\widehat x^H \bar B\widehat x}}.
\end{align}
From this, the rigorous error bounds of the smallest singular value can be ascertained such that $\underline\sigma_{\min}\leq \sigma_{\min}\leq\overline\sigma_{\min}$.

\section{Numerical examples}\label{sec:numex}

In this section, we compare the computing performance of two methods with respect to speed, stability, and accuracy using dense matrices.
The two methods included in the comparison are
\begin{itemize}
    \item[P  :] the previous method using Algorithm \ref{alg:prework} 
    \item[D  :] the proposed method using Algorithms \ref{alg:GHSVD} and \ref{alg:veriGHSVD}
\end{itemize}
The test matrices are those supported by MATLAB.
One of the matrices is a pseudo-random matrix {\tt randn}($n$), where $n$ is the dimension of the matrix.
The other is a pseudo-random matrix with specified singular values, such as
\begin{align}
    \mathtt{gallery}(\mathtt{'randsvd'},n,\mathtt{kappa},3,n-1,n-1,1),
\end{align}
where {\tt kappa} is the specified condition number of $A$ with $\sigma_{\max}(A)\approx 1$.
For details, see \cite{higham2002accuracy}.
We used MATLAB/INTLAB~\cite{Ru99a} for all interval arithmetic.
As the computer environment, we used the supercomputer system Genkai at Kyushu University with one node:
\begin{description}[labelwidth=11em]
   \item[CPU] Intel Xeon Platinum 8490H $\times$ 2 (60 $\times$ 2 cores)
   \item[Amount of Memory] DDR5 4800 MHz, 512 GiB
   \item[Memory Bandwidth] 629 TB/sec
   \item[Software] MATLAB 2024a, INTLAB Ver. 12
\end{description}
The theoretical peak performance is 3{,}456 GFLOPS in double precision.

The relative error for $\sigma_{\min}^{-1}$ satisfying $\underline \rho\leq \sigma_{\min}^{-1}\leq \overline \rho$ denotes
\begin{align}
    \frac{\overline \rho-\underline \rho}{\overline \rho+\underline \rho}\label{eq:relative}.
\end{align}
To estimate the performance of the methods, we compared the best score in ten iterations.

\subsection{Example 1}
In the example, the $n$-by-$n$ real matrices $A$ and $B$ were generated as follows:
\begin{lstlisting}[style=Matlab-editor]
A = randn(n);
B = randn(n); B=n*eye(n)+(B+B')*0.5;
\end{lstlisting}
\begin{table}[htbp]
 \caption{Condition number averages for test matrices $A$ and $B$ (10 runs)}
 \label{tab:cond}
 \centering
  \begin{tabular}{c|cccccc}
   \hline
    $n$           &1000 & 2500 & 5000 & 10000 & 20000 \\\hline \hline
    $\kappa_2(A)$ & 1.27e+03 & 2.94e+03 & 1.15e+04 & 2.09e+04 & 3.92e+04\\
    $\kappa_2(B)$ & 1.09 & 1.06 & 1.04 & 1.03 &1.02\\ 
   \hline
  \end{tabular}
  \vspace{0.5cm}
  \caption{Comparison of relative errors~\eqref{eq:relative} and elapsed times [sec] of the verification methods}
  \label{tab:ex1-1}
  \centering
  \begin{tabular}{r|cc|rr|c}
    \hline
       & \multicolumn{2}{|c}{Relative error of $\sigma_{\min}^{-1}$} & \multicolumn{3}{|c}{Elapsed times [sec]}\\
    \multicolumn{1}{c|}{$n$}    & \multicolumn{1}{c}{P} & \multicolumn{1}{c|}{D} & \multicolumn{1}{c}{P} & \multicolumn{1}{c|}{D} & Ratio (P/D) \\\hline
    2,500    & 5.12e-06 & 3.85e-07 & 26.46 & 1.34 & 18.80\\
    5,000    & 5.13e-06 & 3.80e-06 & 202.33 & 7.44 & 24.93\\
    10,000   & 5.13e-06 & 4.54e-06 & 3,285.43 & 45.48 & 48.20\\
    20,000   & 5.13e-06 & 4.84e-05 & 26,964.60 & 237.38 & 65.79\\
    \hline
  \end{tabular}
\end{table}

Table \ref{tab:cond} shows the condition number averages for the test matrices $A$ and $B$.
As shown by the values in the tables, it is expected that $\kappa_2(A)\sim\mathcal{O}(n)$ and that $B$ is a well-conditioned matrix.
These same matrices are used for the other examples.
Using the matrices, we compared computing performance in terms of accuracy and speed.

From Table~\ref{tab:ex1-1}, the collaborations of the proposed decomposition and verification methods are faster than previous methods.
However, the accuracy of the proposed method deteriorates as the dimension of the matrix increases.
On the other hand, the accuracy reported in previous studies is roughly constant.

\subsection{Example 2}

In this example, the $n$-by-$n$ real matrices $A$ and $B$ were generated as follows:
\begin{lstlisting}[style=Matlab-editor]
A = gallery('randsvd',n,cnd,mode,n-1,n-1,1);
B = randn(n); B=n*eye(n)+(B+B')*0.5;
\end{lstlisting}

\begin{table}[hbtp]
  \caption{Comparison of the relative errors~\eqref{eq:relative} of the verification methods ($n=1000$)}
  \label{tab:ex2-1}
  \centering
  \begin{tabular}{c|cc}
    \hline
       & \multicolumn{2}{|c}{Relative error of $\sigma_{\min}^{-1}$}\\
    \multicolumn{1}{c|}{$\kappa_2(A)$}    & \multicolumn{1}{c}{P} & \multicolumn{1}{c}{D} \\\hline
    \hline
    $10^1$       & 5.34e-06 & 7.71e-11 \\
    $10^3$       & 5.22e-06 & 4.09e-09 \\
    $10^6$       & 5.17e-06 & 2.87e-06 \\
    $10^9$       & 5.16e-06 & 2.34e-03 \\
    $10^{12}$    & 1.34e-05 & inf \\\hline
  \end{tabular}
\end{table}

Table~\ref{tab:ex2-1} shows the relative errors of the verification methods.
The notation `inf' indicates that the enclosure of $\sigma_{\min}$ includes a non-positive number.
Method P is able to enclose the minimum singular value with a high degree of accuracy even when the condition number of matrix $A$ increases.
On the other hand, Method D is less accurate than Method P, while also increasing the condition number of matrix $A$.
The proposed method is highly accurate when the condition numbers of both matrices $A$ and $B$ are low.
However, this may depend on the way the norm function is implemented in INTLAB.
Still, even taking this into account, the proposed method is superior.

\subsection{Example 3}
In this example, the $n$-by-$n$ real matrices $A$ and $B$ were generated as follows:
\begin{lstlisting}[style=Matlab-editor]
A = randn(n);
B = gallery('randsvd',n,-cnd,mode,n-1,n-1,1);
\end{lstlisting}

\begin{table}[hbtp]
  \caption{Comparison of relative errors~\eqref{eq:relative} of the verification methods}
  \label{tab:ex3-1}
  \centering
  \begin{tabular}{c|cc}
    \hline
       & \multicolumn{2}{|c}{Relative error of $\sigma_{\min}^{-1}$}\\
    \multicolumn{1}{c|}{$\kappa_2(B)$}    & \multicolumn{1}{c}{P} & \multicolumn{1}{c}{D} \\\hline
    \hline
    $10^1$       & 2.11e-02 & 5.26e-08 \\
    $10^3$       & failed & 3.14e-07 \\
    $10^6$       & failed & 4.69e-04 \\
    $10^9$       & failed & 1.15e-01 \\
    $10^{12}$    & failed & inf \\\hline
  \end{tabular}
\end{table}

Table~\ref{tab:ex3-1} shows the relative errors of the verification methods.
The notation `failed' indicates that the calculation of the interval Cholesky decomposition failed due to the square root calculation for a non-positive number.
It can be seen that the proposed method is able to stably encompass the minimum singular value even when the condition number of matrix $B$ increases.
The interval Cholesky decomposition may also fail even if the condition number of the matrix is small.
From this, it can be said that the proposed method is more stable than that offered in previous studies.

\section{Application to differential equations}\label{sec:PDE}

Let $\Omega \subset \mathbb{R}^d$ be a convex bounded polygonal ($d=1,2$), and for some integer $m$, let $H^m(\Omega)$ denote the complex $L^2$-Sobolev space of order $m$ on $\Omega$. 
We define the Hilbert space
$
   H^1_0 := \{ u \in H^1(\Omega) \ | \ u=0 \ \text{on} \ \partial \Omega \}.
$
Consider the linear elliptic operator
\begin{equation}\label{sec5-L}
  \mathscr{L} u := -\Delta u + b\cdot \nabla u + c u: \quad  H^2(\Omega)\cap H^1_0(\Omega) \to L^2(\Omega)
\end{equation}
for $b \in L^\infty(\Omega)^d$, $c \in L^\infty(\Omega)$.
The operator $\mathscr{L}$ is a general expression of the linearized operator of a nonlinear problem to find $u \in H^1_0(\Omega)$
satisfying $-\Delta u = f(x,u,\nabla u)$ with a bounded operator $f: H^1_0(\Omega) \to L^2(\Omega)$; in order to apply some fixed-point theorem, the invertibility of $\mathscr{L}$ and its computable bound of the operator norm of $\mathscr{L}^{-1}$ are essential tasks for computer-assisted proofs of $u \in H^1_0(\Omega)$.

Let $S_h$ be a finite dimensional approximation subspace of $H^1_0(\Omega)$
dependent on the parameter $h>0$. For example, $S_h$ is taken to be a
finite element subspace with mesh size $h$.
We define the basis function of $S_h$ by $\{\phi_i\}_{i=1}^N$ for $N := \dim S_h$
and $N \times N$ matrices $A$ and $B$ by
\begin{align}
   [A]_{ij} &= (\nabla\phi_j,\nabla\phi_i)_{L^2} 
               + (b\cdot\nabla\phi_j+c \phi_j,\phi_i)_{L^2}, \label{matrix_A} \\
   [B]_{ij} &= (\nabla\phi_j,\nabla\phi_i), \label{matrix_B}
\end{align}
respectively, where $(\cdot,\cdot)_{L^2}$ is the $L^2$-inner product on $\Omega$.
Note that the matrix $B$ is positive definite.
Tables~\ref{tab:ex1_R=5_C=-15} and~\ref{tab:ex1_R=6.75_C=-1-1.5i}
show the verification results in \eqref{sec5-L} for
\[
   b(x_1,x_2) = R \begin{bmatrix}-x_2+0.5 \\ x_1-0.5 \end{bmatrix}
\]
with $(R,c)=(5,-15)$ and for $(R,c)=(6.75,-1-1.5i)$, 
respectively, which comes from a stationary convection-diffusion equation.
%

\begin{itemize}
    \item[P  :] the previous method with Algorithm \ref{alg:prework} 
    \item[D  :] the proposed method with Algorithm \ref{alg:veriGHSVD}
    \item[S1 :] the proposed method with Algorithm \ref{alg:veriGHSVD_type2}
    \item[S2 :] the proposed method with Algorithm \ref{alg:veriGHSVD_type3}
\end{itemize}



\begin{table}[hbtp]
  \caption{Upper bounds of $\sigma_{\min}^{-1}$ and the computation times of the verification methods. The matrices are given by \eqref{matrix_A} and \eqref{matrix_B} with $(R,c)=(5,-15)$.}
  \label{tab:ex1_R=5_C=-15}
  \centering
  \begin{tabular}{r|cccc|rrrr}
    \hline
      & \multicolumn{4}{|c}{Upper bounds $\sigma_{\min}^{-1}$} & \multicolumn{4}{|c}{Elapsed times [sec]}\\
    \multicolumn{1}{c|}{$n$}    & P & D & S1 & S2 & \multicolumn{1}{c}{P} & \multicolumn{1}{c}{D} & \multicolumn{1}{c}{S1} & \multicolumn{1}{c}{S2} \\\hline
    841    & 4.1234 & 4.1233 & 5.8512 & 4.1233 & 6.79 & 0.98 & 2.43 & 1.08\\
    9,801   & 4.1555 & 4.1555 & 5.9023 & 4.1555 & 298.92 & 100.94 & 4.74 & 0.98 \\
    89,401  & - & - & failed & 4.1625 & - & - & failed & 10.65 \\
    998,001 & - & - & failed &  4.1628 & - & - & failed & 1408.68 \\
    \hline
  \end{tabular}
\end{table}

\begin{table}[hbtp]
  \caption{Upper bounds of $\sigma_{\min}^{-1}$ and the computation times of the verification methods. The matrices are given by \eqref{matrix_A} and \eqref{matrix_B} with $(R,c) = (6.75,-1-1.5i)$.}
  \label{tab:ex1_R=6.75_C=-1-1.5i}
  \centering
  \begin{tabular}{r|cccc|rrrr}
    \hline
      & \multicolumn{4}{|c}{Upper bounds $\sigma_{\min}^{-1}$} & \multicolumn{4}{|c}{Elapsed times [sec]}\\
    \multicolumn{1}{c|}{$n$}    & P & D & S1 & S2 & P & D & S1 & S2 \\\hline
    841    & 1.0496 & 1.0495 & 1.5436 & 1.0496 & 29.65 & 4.28 & 7.75 & 1.12\\
    9,801   & 1.0497 & 1.0497 & 1.5597 & 1.0497 & 983.32 & 339.53 & 11.82 & 3.08 \\
    89,401  & - & - & failed & 1.0497 & - & - & failed & 46.79 \\
    998,001 & - & - & failed &  1.0500 & - & - & failed & 2340.33 \\
    \hline
  \end{tabular}
\end{table}

The results shown in Tables \ref{tab:ex1_R=5_C=-15} and \ref{tab:ex1_R=6.75_C=-1-1.5i} indicate that the proposed method effectively maintains accuracy and speed for both real and complex coefficient matrices. While Method S1 operates quickly, it faces challenges in numerical stability and precision when applied to large-scale problems. On the other hand, Method S2 can handle large-scale problems that were previously unmanageable with traditional methods, and it does so with considerable speed. The numerical stability of Method S2 is also satisfactory. 

\section{Conclusion}
This paper discusses a verification method for the singular values of \( RA^{-1}R^H \) and proposes both matrix decomposition techniques and their associated verification methods. 
Two distinct methods are introduced:
\begin{enumerate}
    \item A high-stability method, which computes a lower bound of the minimum singular value more accurately than previous methods in practical applications.
    \item A fast method designed for large-scale sparse matrices, which can be applied to real matrices with dimensions of approximately \( n \approx 1,000,000 \).
\end{enumerate}
Both methods show improved efficiency compared to the methods proposed in previous studies.

\section*{Acknowledgments}
This work was supported by Grants-in-Aid from the Ministry of Education, Culture, Sports, Science and Technology of Japan (Nos. 23K28100 and 24H00694).
The computations were mainly carried out using the computer facilities at the Research Institute for Information Technology, Kyushu University, Japan.
 
 \bibliographystyle{elsarticle-num} 
 \bibliography{cas-refs}
\end{document}